\documentclass[12pt,a4paper]{article}

\usepackage{amsmath, amssymb, amsthm}
\usepackage{geometry}
\usepackage{graphicx}
\usepackage{mathrsfs}
\usepackage{enumitem}
\usepackage{cite}
\usepackage{float}     
\usepackage{url}
\usepackage[colorlinks=true, linkcolor=blue, citecolor=blue, urlcolor=blue]{hyperref}

\newtheorem{theorem}{Theorem}[section]
\newtheorem{lemma}[theorem]{Lemma}

\theoremstyle{definition}

\theoremstyle{remark}
\newtheorem{remark}[theorem]{Remark}
\newenvironment{acknowledgment}{
	\section*{Acknowledgment}
}{}

\title{
	Minimal Perimeter Triangle in Nonconvex Quadrangle: \\
	\textit{Generalized Fagnano Problem}
}
\author{
	Triloki Nath\thanks{
		Department of Mathematics \& Statistics, Deen Dayal Upadhyaya Gorakhpur University, Gorakhpur, Uttar Pradesh (INDIA). \texttt{triloki.mathstat@ddugu.ac.in}
	}
	\and 
	Manohar Choudhary\thanks{
		Department of Mathematics \& Statistics, Dr Harisingh Gour Vishwavidyalaya, Sagar, Madhya Pradesh (INDIA). \texttt{y21472007.rs@dhsgsu.edu.in}
	}
}
\date{}

\begin{document}
	
	\maketitle
	
	\begin{abstract}
		In 1775, Fagnano introduced the following geometric optimization problem: inscribe a triangle of minimal perimeter in a given acute-angled triangle.	
		A widely accessible solution is provided by the Hungarian mathematician L. Fejer in 1900.
		This paper presents a specific generalization of the classical Fagnano problem, which states that given a nonconvex quadrangle (having one reflex angle and others are acute angles), find a triangle of minimal perimeter with exactly one vertex on each of the sides that do not form reflex angle, and the third vertex lies on either of the sides forming the reflex angle. We provide its geometric solution. Additionally, we establish an upper bound for the classical Fagnano problem, demonstrating that the minimal perimeter of the triangle inscribed in a given acute-angled triangle cannot exceed twice the length of any of its sides.
	\end{abstract}

 

\section{Introduction}

For centuries, the field of geometric optimization has been fascinated by the beauty of geometry, inspired by timeless questions of geometric extrema, focusing on minimizing or maximizing fundamental measures such as area, volume, and perimeter. One of the earliest milestones in geometric optimization originated with \textit{Heron of Alexandria}, who formulated the problem of finding the shortest path from a point on a given line to two fixed points in the plane. Though deceptively simple in appearance, this problem, now celebrated as \textit{Heron’s problem}, introduced the foundational idea of path minimization and anticipated the geometric principle, the \textit{law of reflection,} that would later shape much of classical optimization.

As mathematics evolved, the spirit of geometric exploration grew richer, leading to ever more refined and elegant formulations. In the seventeenth century, the \textit{Fermat–Torricelli problem} carried Heron’s idea from paths to points, asking for a location that minimizes the total distance to three given vertices. This shift from linear to planar optimization represented a significant conceptual advance, revealing a deep interplay between geometry and minimality.

A century later, the pursuit of geometric perfection reached a new level of beauty and insight with the \textit{Fagnano problem}, which asked for the triangle of smallest perimeter that could be inscribed within an acute-angled triangle. This marked a shift in focus—from identifying a single point of minimal distance to determining an entire geometric configuration that satisfies optimality under given constraints. Viewed collectively, these problems trace the expansion of geometric optimization into the wider field of constrained optimization—an evolution from tracing the shortest path, governed by the law of reflection, to constructing the most efficient form, each guided by the same enduring pursuit of elegance through precision.

\section{Classical Fagnano's Problem:}  
The Fagnano's problem, proposed by the Italian mathematician Giulio Carlo Toschi di Fagnano is solved by his son Giovanni Francesco Fagnano \cite[p. 65]{nahin_when_2007}. It is also referred to as the Schwarz triangle problem \cite[p. 306]{giaquinta_mathematical_2003}.
The problem is to find a triangle of minimal perimeter inscribed in a given acute-angled triangle $\triangle ABC$. The solution of this problem is orthic triangle or altitude triangle $\triangle PQR$ (see Figure \ref{orthictriangle}), that is vertices $P, Q, R$ are the feet of perpendiculars drawn from vertices $A, B, C$ to its opposite sides respectively, see e.g. \cite[p. 88]{coxeter_geometry_2005}, \cite[p.3]{andreescu_geometric_2006}.
\begin{figure}[H] 
	\centering
	\includegraphics[height=5cm]{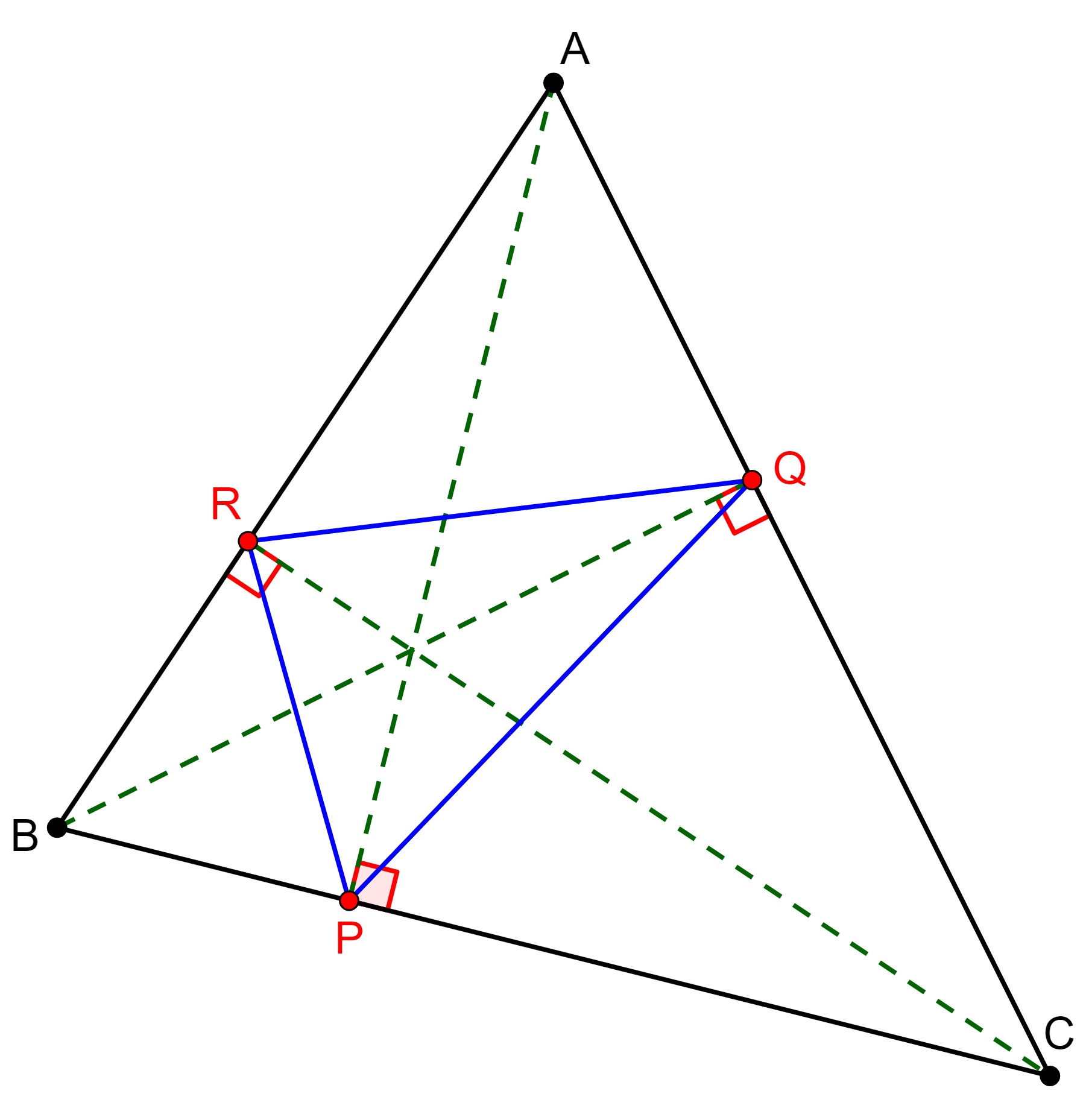} 
	\caption{Illustrating the solution to classic Fagnano's problem}
    \label{orthictriangle}
\end{figure}
\begin{theorem}[\textbf{Fagnano}]\label{classicalFagnanoTheorem}
	The triangle of minimal perimeter inscribed in a given acute-angled triangle is the orthic triangle.
\end{theorem}

The geometric solution to Fagnano’s problem, as provided by  Lip\'{o}t Fej\'{e}r, is elaborated in \cite[p. 3]{andreescu_geometric_2006}.
It is much easier than the approach of Hermann Amandus Schwarz \cite[p. 88]{coxeter_geometry_2005}, \cite[p. 346]{courant_what_1996}, who proved the minimal property of the orthic triangle of an acute-angled triangle using the principle of reflection in edges repetitively five times. He also proved that this holds true even when the given triangle is not acute-angled, say, $\angle BAC \geq 90^{\circ}.$ Using Fej\'{e}r’s approach, it is not difficult to see that the points $Q$ and $R$ will coincide with $A$. In other words, in this case, $\triangle PQR$ is degenerate. 

\subsection{Modern Significance and Applications}

Although these problems originated in purely geometric curiosity, their modern relevance extends far beyond classical geometry. Optimization of distances and perimeters under boundary constraints naturally arises in \textit{location science}, \textit{facility planning}, and \textit{network design}, where the goal is to minimize total connection length among designated sites. Similarly, in \textit{robotic navigation} and \textit{path planning}, geometric optimization principles govern motion efficiency and obstacle avoidance. Thus, studying such problems not only enriches geometry but also informs diverse practical applications where optimal geometric configurations are required.

\subsection{From the Classical to the Generalized Fagnano Problem}
Motivated by the structure of the Fagnano problem, recent research has sought to extend its 
principles to polygons beyond triangles. In particular, a generalization to \emph{convex 	quadrangle}, where one seeks the orthic quadrangle of minimal perimeter with vertices on 
the sides of a convex quadrangle has been proposed \cite{Mammana2010Orthic}. This extension 
preserves the geometric essence of the original problem while introducing new analytic challenges.

However, for \emph{nonconvex quadrangle}, it is geometrically impossible to construct a 
quadrangle with all four vertices on the sides such that the quadrangle is entirely 
contained within the quadrangle, since one of its interior angles exceeds $180^{\circ}$. 
Consequently, the problem must be reformulated: instead of seeking a quadrangle, we construct a triangle with vertices on the sides of the nonconvex quadrangle, where one vertex lies on each of the two sides that do not form the reflex angle, and the third vertex lies on either of the two sides forming the reflex angle.

\medskip
\textbf{Notation.} Throughout this paper, the following conventions are used:
\begin{itemize}
	\item For points $A$ and $B$, the symbol $AB$ denotes the closed line 
	segment with endpoints $A$ and $B$ (both endpoints included).
	
	\item The length of segment $AB$ is denoted by $|AB|$.

	\item  When we say a point $X$ lies \emph{between} $AB$, we mean $X$ 	belongs to the interior of segment $AB$, that is, $X$ lies on $AB$ with  	$X \neq A$ and $X \neq B$. 
	
	\item The perimeter of $\triangle ABC$ is denoted by 
	$\mathrm{per}(\triangle ABC)$.
\end{itemize}

\section{The Generalized Fagnano Problem} \label{Generalized Fagnano Problem}
Now, we propose the following generalization of the classical Fagnano problem:
\begin{quote}
	Let $ABCD$ be a quadrangle with a reflex angle at $C$  and all other angles acute. Determine a triangle of minimal perimeter with one vertex 
	on each of the sides that do not form the reflex angle, and the third vertex  lying on one of the sides forming the reflex angle.
\end{quote}

 Symbolically: given a quadrangle $ABCD$ with a reflex angle at $C$ and others are acute angles(see Figure \ref{fig:quadrangle}). Our goal is to find points $P$ on $AB$, $Q$ on $AD$, and $R$ on either $BC$ or $CD$ such that $\triangle PQR$ has minimal perimeter. Furthermore, observe that if $\angle BCD=180^\circ$ (i.e., points $B$, $C$, $D$ are collinear rather than forming a reflex angle at $C$), then the generalized Fagnano problem for quadrangle $ABCD$ reduces to the classical Fagnano problem for $\triangle ABD$. 
The solution is either the orthic triangle or degenerate, depending on whether $\triangle ABD$ is acute-angled or not.

\begin{figure}[H]
	\centering
	\includegraphics[height=5cm]{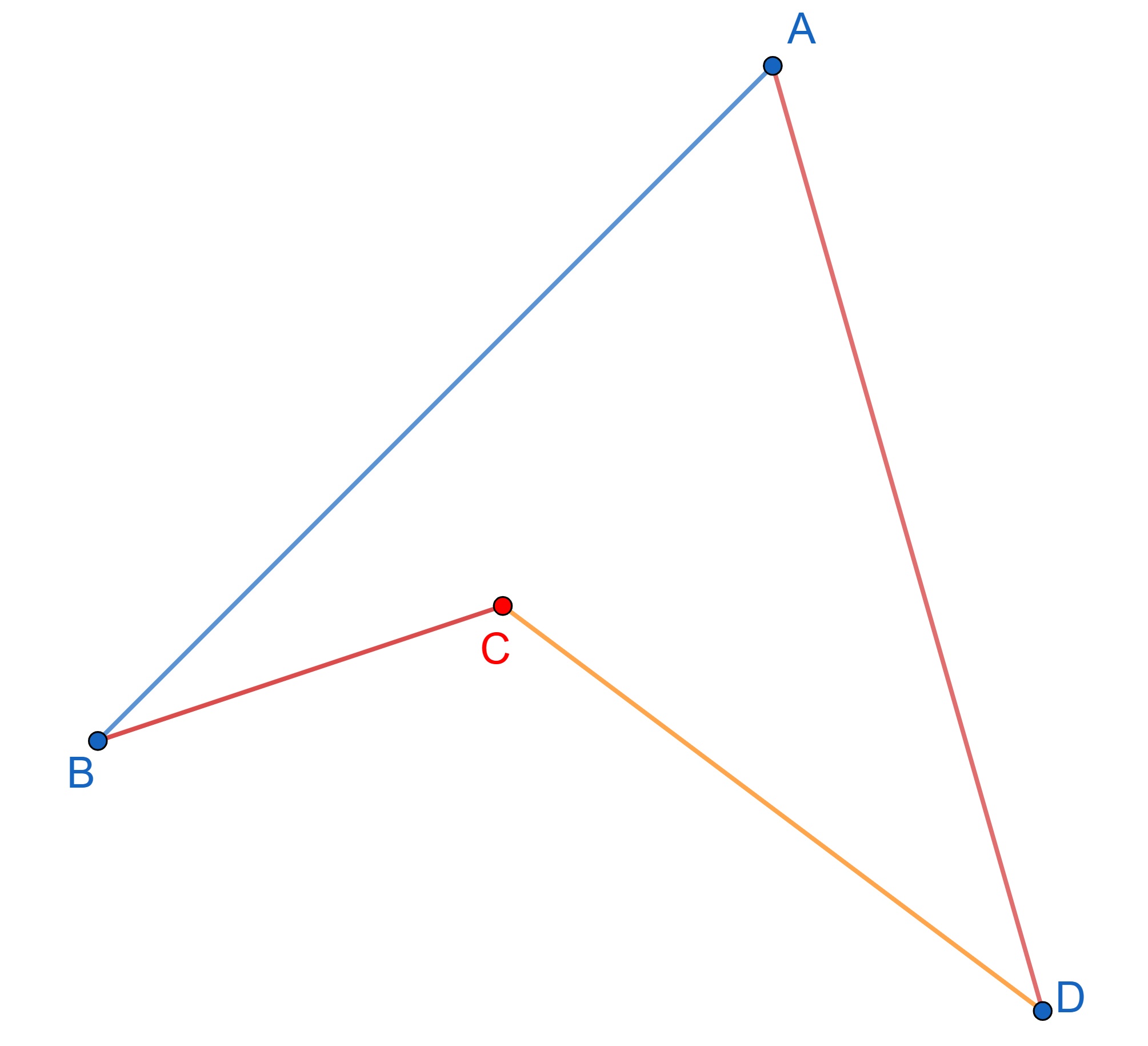}
	\caption{Acute-angled quadrangle $ABCD$ with reflex angle at $C$.}
	\label{fig:quadrangle}
\end{figure}

In what follows, we establish a complete geometric solution to this generalized problem. 
Specifically, we show that the triangle of minimal perimeter satisfying the above conditions 
corresponds to the orthic triangle of a certain auxiliary triangle constructed from the 
given quadrangle. This result not only connects the classical and generalized formulations 
but also extends the theoretical and practical reach of geometric optimization to a broader 
class of problems.

\section{Solution of the Generalized Fagnano Problem} 
The following theorem characterizes the solution of the generalized Fagnano problem for 
the quadrangle stated in Section~\ref{Generalized Fagnano Problem}.

\begin{theorem}\label{maintheorem}
	Let $ABCD$ be a quadrangle with a reflex angle at $C$ and others are acute angles. Consider the line segment $AC$ and draw the perpendicular to $AC$ through $C$. The triangle of minimal perimeter with 	one vertex on $AB$, one vertex on $AD$, and the third 
	vertex on either $BC$ or $CD$ is determined by the 	following exhaustive cases:
	
	(A) If the perpendicular intersects sides $AB$ and $AD$ at points $M$ and $N$ respectively, then the 	orthic triangle of $\triangle AMN$ is the triangle 	of minimal perimeter.
	
	(B) Otherwise, without loss of generality assume that the perpendicular does not meet $AB$. Extend side $BC$ 	to meet $AD$ at point $W$. Then the orthic triangle	of $\triangle ABW$ is the triangle of minimal	perimeter.
\end{theorem}

Before proceeding to the proof of Theorem~\ref{maintheorem}, it is important to note that 
the points $M$ and $N$ defined above may or may not lie on the sides $AB$ and $AD$ 
respectively. Their existence depends on the angular constraints at vertices $B$ and $D$. 
Specifically, the points $M$ and $N$ exist if and only if the following inequalities hold:
\begin{equation}\label{existence}
	\angle B \leq 90^{\circ} - \angle BAC, 
	\qquad 
	\angle D \leq 90^{\circ} - \angle DAC.
\end{equation}
When either of these conditions fails, Case~\textbf{(B)} naturally arises.

For better understanding and visualization of Case~\textbf{(B)}, the behavior and existence 
of the points $M$ and $N$ can be explored interactively using GeoGebra.\footnote{Available 
	online at: \url{https://www.trilokinath.in/p/fnc.html}}

\medskip

The proof of Theorem~\ref{maintheorem} requires a few fundamental geometric facts about triangles. To keep our exposition self-contained and accessible, we provide proof below  along with a key observation about the minimal triangle.

\textbf{Fact 1:} Given an acute-angled $\triangle AMN$, extend $AM$ to $AX$ and $AN$ to $AY$. Then, $|XY| \geq |MN|$.

\begin{proof}
	\begin{figure}[H]
		\centering
		\includegraphics[height=5cm]{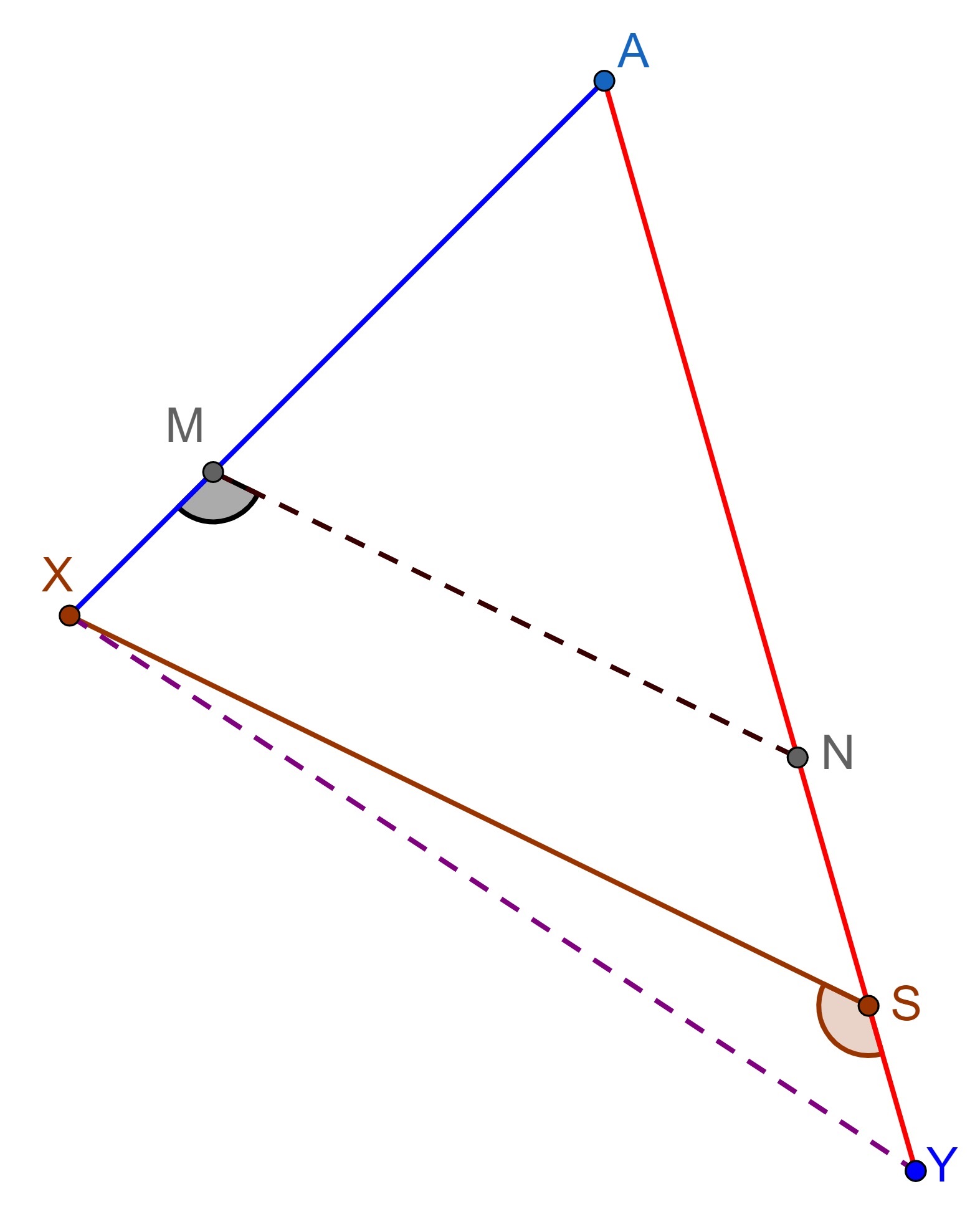}
		\caption{Illustration of Fact 1}
        \label{Fact1}
	\end{figure}
	Without loss of generality, assume that the line parallel to $MN$ through $X$ intersects the segment $NY$ at $S$(see Figure \ref{Fact1}). Then by similarity of triangle, $|XS| \geq |MN|.$ In $\triangle XSY,$ note that $\angle XSY >90$\textdegree, whereby $|XY| \geq |XS|.$ Consequently,
	\begin{equation}
		|XY| \geq |MN|,  \label{inequalities01}
	\end{equation}
	holds with strict inequality if $ X \neq M$ or $ Y \neq N$.
\end{proof}
\textbf{Fact 2:} Given any $\triangle ABC$, consider an arbitrary point $ E $ inside $ \triangle ABC $. Then,
\begin{equation}\label{perimeter_remark}
	\text{per} (\triangle ABC) > \text{per} (\triangle EBC).
\end{equation}
\begin{proof}
	This observation can be viewed in the following way:
	Since $E$ is an arbitrary point inside $\triangle ABC$, draw a perpendicular from $A$ to $BC$ (largest side), intersecting $BC$ at $M$. Join $BE$ and $CE,$ then at least one intersects the perpendicular AM. Suppose $BE$ intersects AM. Now, extend CE to intersect AM at the point $D$ (see Figure \ref{fig:fact2}). Join $BD, $ since $\angle ABM + \angle BAM=90$\textdegree, $\angle ABD < \angle ABM, \angle ADB > 90$\textdegree, similarly, $\angle ADC > 90$\textdegree.
	
	\begin{figure}[H]
		\centering
		\includegraphics[height=5cm]{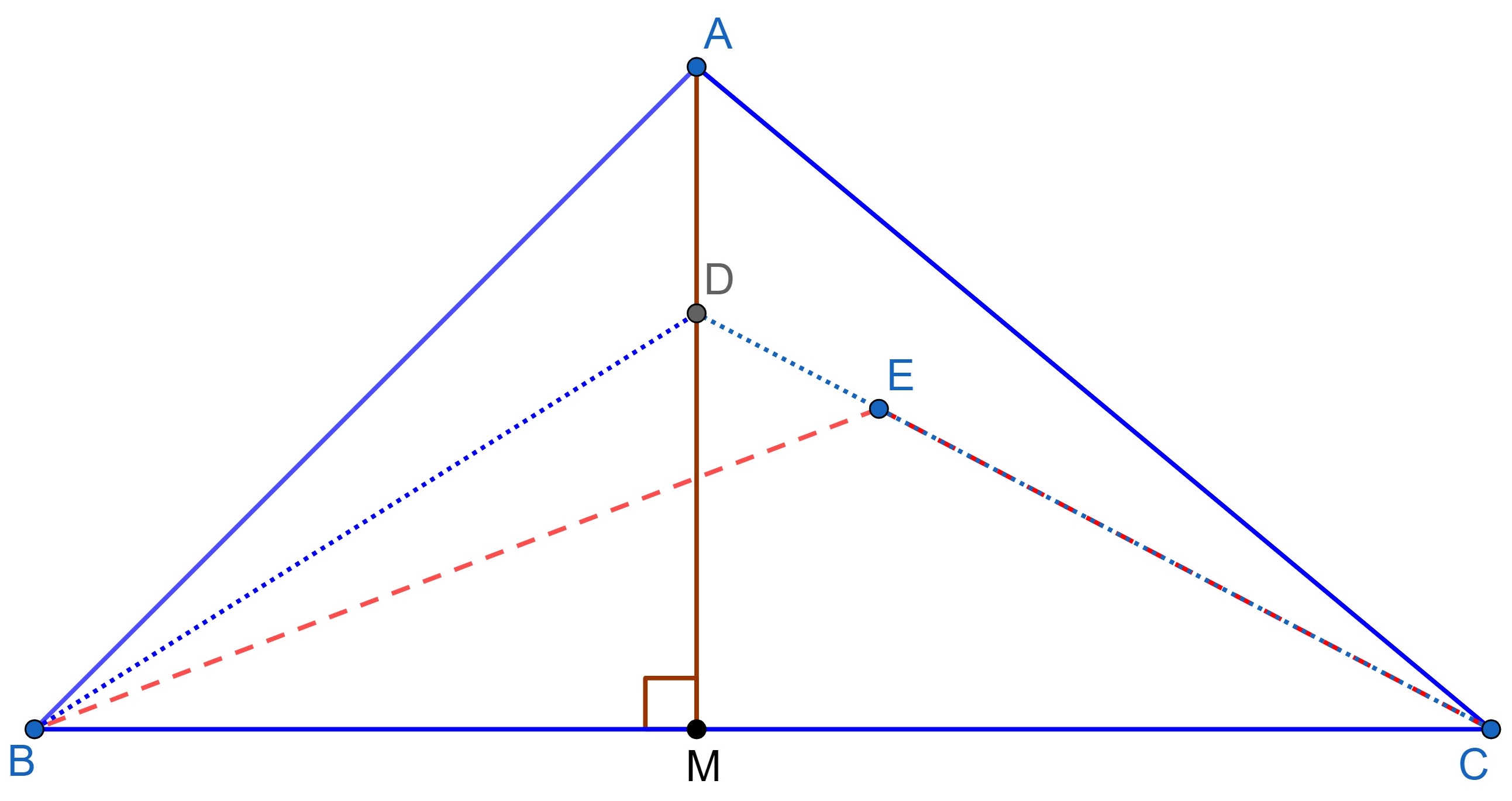}
		\caption{Illustration of Fact 2}
		\label{fig:fact2}
	\end{figure}
	Therefore, by the triangle inequality,
\begin{align*}
|AB| + |AC| &> |BD| + |DC| \\
            &= |BD| + (|DE| + |EC|) \\
            &= (|BD| + |DE|) + |EC| \\
            &\geq |BE| + |EC| \quad \text{(since $E$ and $D$ may coincide)}
\end{align*}
    
	Hence (\ref{perimeter_remark}) follows.
\end{proof}

\begin{lemma} \label{lemma01}
	The minimal perimeter of the triangle inscribed in a given acute-angled triangle cannot exceed twice the length of the smallest side.
\end{lemma}

\begin{proof}
	We begin by assuming that $\triangle PQC$ has the minimal perimeter among all the triangles inscribed in the acute-angled $\triangle AMN$. It is known that $\triangle PQC$ is an orthic triangle, which is a solution to the classical Fagnano problem. 
	\begin{figure}[H]
		\centering
		\includegraphics[height=6cm]{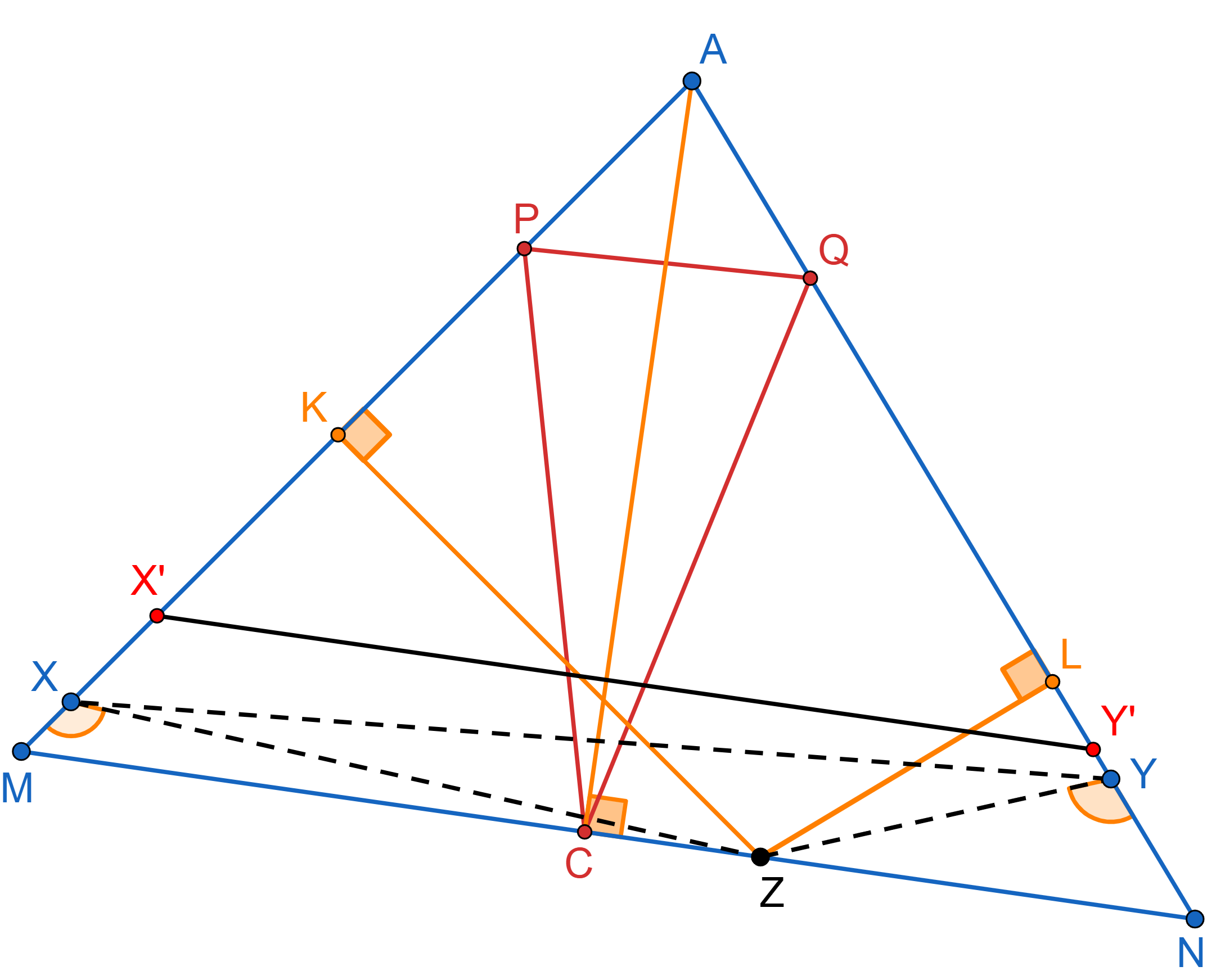}
		\caption{Illustration of Lemma \ref{lemma01}}
	\end{figure}
	
	Consider points   $ X   $,   $ Y   $, and   $ Z   $ such that   $ X   $ lies on the side   $ AM   $,   $ Y   $ lies on the side   $ AN   $, and   $ Z   $ lies on the side  $MN$. By the classical solution to the Fagnano problem, the perimeter of triangle   $ \triangle PQC   $ is less than or equal to the perimeter of any triangle   $ \triangle XYZ$ for all such $X$, $Y$, and $Z$. This can be written as:
	$$
	\text{per} (\triangle PQC) \leq \text{per} (\triangle XYZ) = |XY| + |XZ| + |YZ|
	$$ 
	
	with equality if and only if $\triangle XYZ$ coincides with 
	$\triangle PQC$.\\

	We now establish the upper bound. Fix the point $Z$ on segment $MN$.  Draw  perpendiculars $ZK$ and $ZL$ on $AM$ and $AN$ respectively. Draw a line $X'Y' || MN,$ where $X'$ lying on $KM$ and $Y'$ on $LN$.
	Since $\triangle AMN$ is acute-angled, so is $\triangle AX'Y'$. 
	Therefore, using \textbf{Fact~1}, for any $X$ between 
	$X'M$, and $Y$ between $Y'N$, we have $|XY| > |X'Y'|$. 
	That is, we can take $X$ and $Y$ so that $|XY|$ permanently increases as $X$ approaches to $M$ and $Y$ approaches to $N$. Also, $|XY| < |MN|$ if 
	$X \neq M$ or $Y \neq N$. Since $\angle ZXM$ and $\angle ZYN$ are obtuse for such $X$ and $ Y$, therefore 
	$|XZ| < |MZ|$ and $|YZ| < |NZ|$  	which  implies that   $|XZ|+ |YZ| < |MN|$. Clearly,  $|XZ|+|YZ|$ approaches to $|MN|$ 	as $X$ approaches to $M$ and $Y$ approaches to $N$.
	
	Thus, we always have $per(\triangle XYZ) < 2|MN|$ for such $X, Y$ and $Z$. Consequently  
	
	\begin{equation} \label{perimeter<2MN}
		\text{per} (\triangle PQC) < 2|MN|.
	\end{equation}
	Hence,
	$\mathrm{per}(\triangle PQC) < 2  \min\{|AM|, |MN|, |AN|\}.$
\end{proof}
Now, we have enough tools to prove our main result. The following is the proof of Theorem \ref{maintheorem}.
\begin{proof}
	We begin the proof of part (A) by noting that, $\angle ACN = 90^\circ$. Draw a perpendicular to $AN$ from $M$ intersecting $AN$ at $Q$. Similarly, draw a perpendicular to $AM$ from $N$ intersecting $AM$ at $P.$ We have to prove that the $\triangle PQC, $ inscribed in the $\triangle AMN$, so obtained has the minimal perimeter to our problem. \\
	Let the minimal perimeter be denoted by $\textbf{$p_F$}$  (i.e., $p_F = |PC| + |PQ| + |QC|$).
	\begin{figure}[H]
		\centering
		\includegraphics[height=5cm]{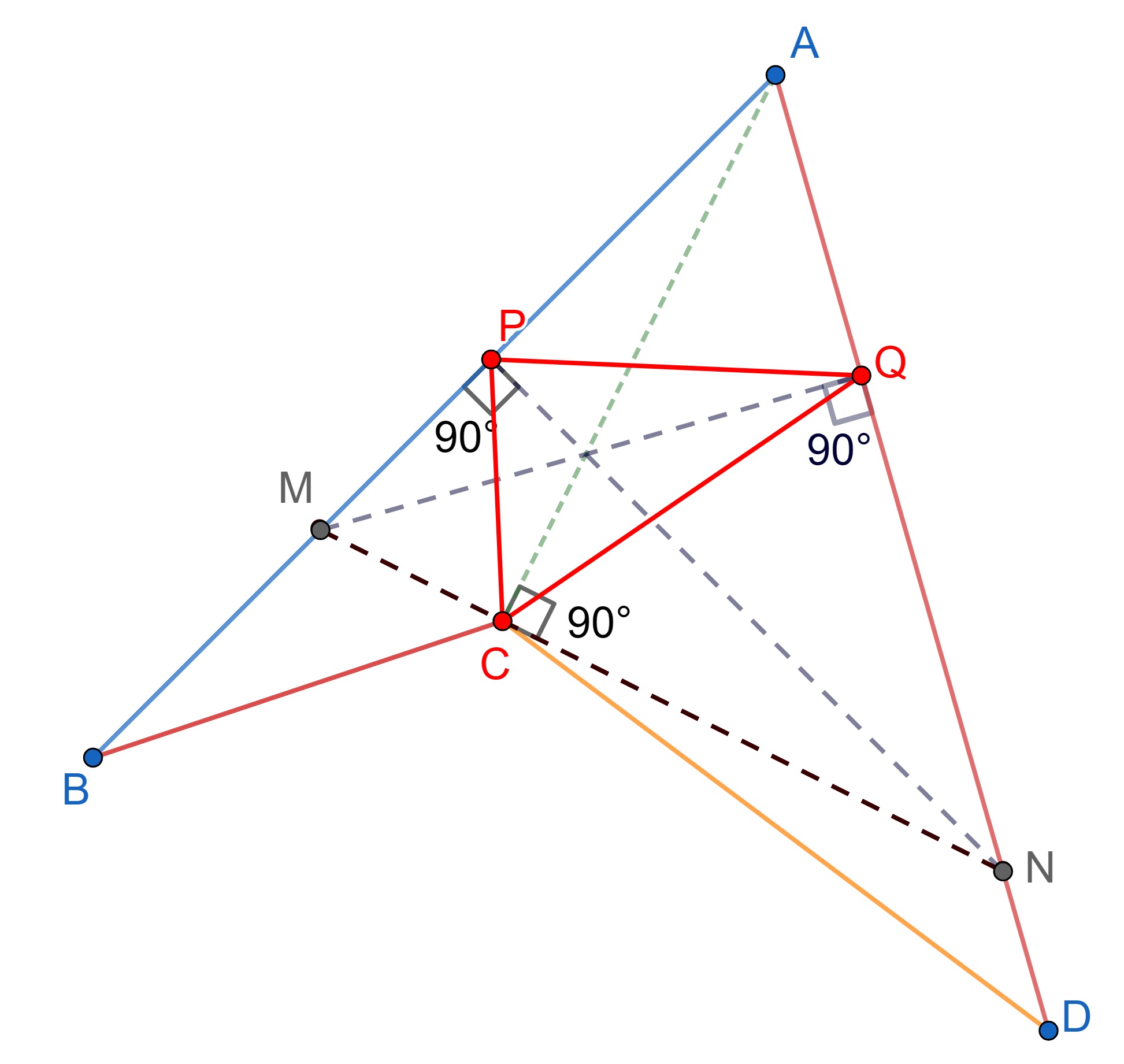}
		\caption{Illustration of solution of generalized Fagnano problem}
	\end{figure}
	To verify our assertion that $\triangle PQC$ has the minimal perimeter, we investigate all exhaustive cases for $\triangle XYZ$ formed in the nonconvex quadrangle $ABCD$, where $X$ is on $AB$, $Y$ is on $AD$ and $Z$ is either on $BC$ or $CD.$ We prove that the perimeter of \( \triangle XYZ \) is not less than that of \( \triangle PQC \).

\begin{description}
	\item[\textbf{Case 1}] Suppose $Z = C$ is fixed, with $X$ lying anywhere 
	on side $AB$ and $Y$ lying anywhere on side $AD$. 
	
	We employ the reflection principle to handle all positions of $X$ and 
	$Y$ simultaneously. Reflect point $C$ in line $AB$ to $C'$, and 
	reflect $C$ in line $AD$ to $C''$ (see Figure~\ref{fig:case2reflection}).
	\begin{figure}[H]
		\centering
		\includegraphics[height=6cm]{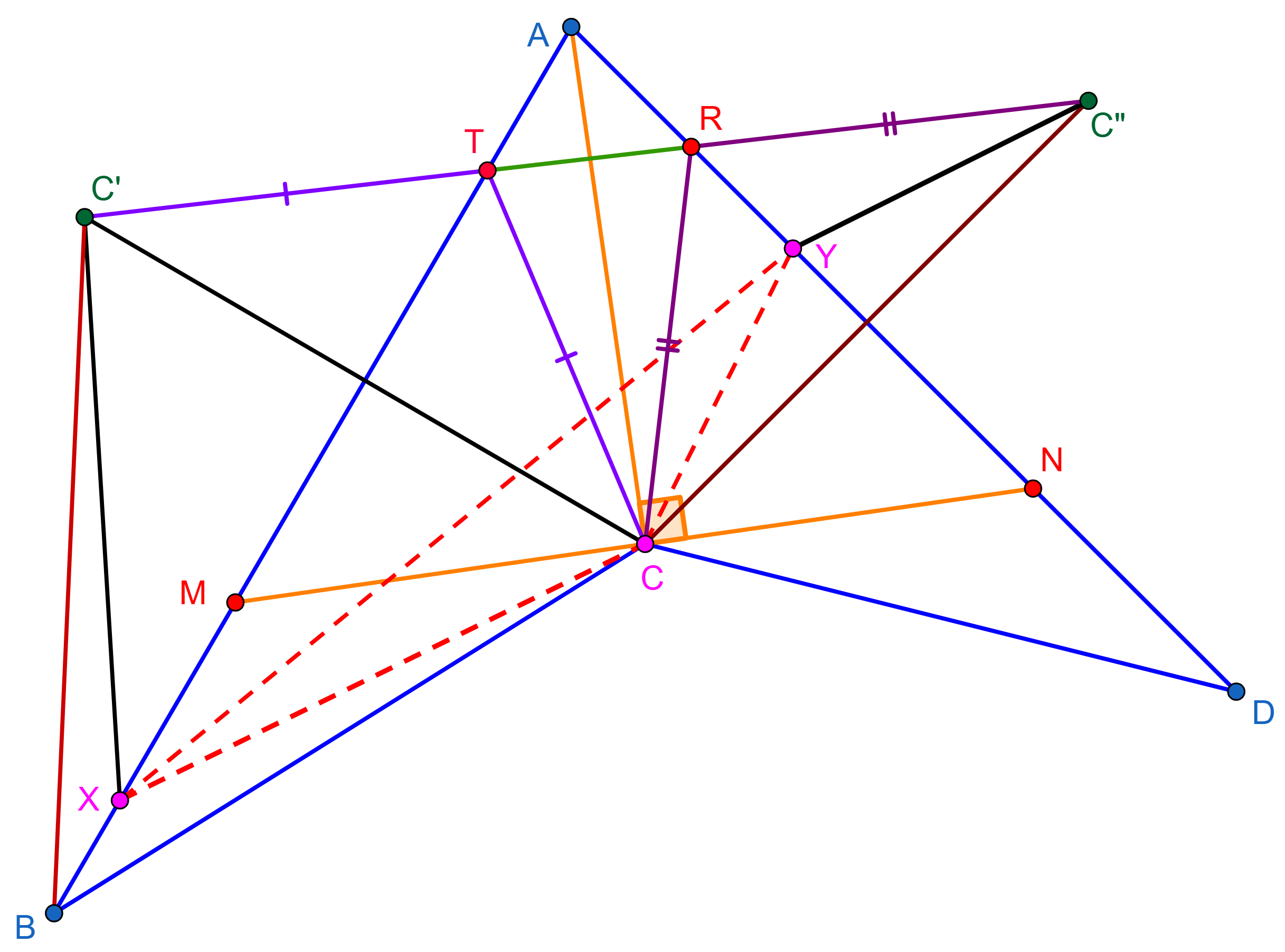}
		\caption{Illustration of Case~1 using the reflection principle.}
		\label{fig:case2reflection}
	\end{figure} 
	By the definition of reflection, for any point $X$ on $AB$ and any point 
	$Y$ on $AD$, we have
	$$|XC| = |XC'| \quad \text{and} \quad |YC| = |YC''|.$$
	
	Therefore, the perimeter of $\triangle XYC$ can be expressed as:
	\begin{equation}\label{eq:reflection_perimeter}
		\mathrm{per}(\triangle XYC) = |XC| + |XY| + |YC| 
		= |XC'| + |XY| + |YC''|.
	\end{equation}
	
	By the triangle inequality,
	\begin{equation}\label{eq:triangle_ineq}
		|XC'| + |XY| + |YC''| \geq |C'C''|,
	\end{equation}
	with equality if and only if $X$ and $Y$ lie on the line segment $C'C''$.
	
	It remains to show that $|C'C''| = p_F$. Since $\triangle AMN$ is acute-angled and $C$ lies on $MN$. The segment $C'C''$ intersects 	sides $AM$ and $AN$ at points $T$ and $R$ respectively. By the reflection 
	principle, we have
	\begin{align*}
		|C'C''| &= |C'T| + |TR| + |RC''|\\
		&= |CT| + |TR| + |RC|\\
		&= \mathrm{per}(\triangle TRC).
	\end{align*}

	Since $C$ is the foot of perpendicular from $A$ to $MN$, therefore $T$ and $R$ are precisely the feet of the perpendiculars from $N$ to side $AM$ and from $M$ to side $AN$ respectively, see,  e.g. \cite[Problem 1.1.3]{andreescu_geometric_2006}, which means $\triangle TRC$ is the orthic triangle of $\triangle AMN$. By Theorem  \ref{classicalFagnanoTheorem}, it must coincide with the orthic $\triangle PQC$ of  $\triangle AMN$.
	
	Combining equations \eqref{eq:reflection_perimeter} and 
	\eqref{eq:triangle_ineq}, we conclude that
	$$\mathrm{per}(\triangle XYC) \geq |C'C''| = p_F,$$
	with equality if and only if $X = P (=T)$ and $Y = Q(=R)$.
	
\end{description}

	\begin{description}
		\item[\textbf{Case 2}] Suppose $Z$ is any point on $CB$ or $CD$ and $X$ and $Y$ are between $AM$ and $AN$ respectively.
		We further divide this case into two sub-cases.  
		\begin{description}
			\item[\textbf{Case 2.1}] Suppose $Z$ lies between $BC$.\\
			For $\triangle AMN$, $\triangle PQC$ is the solution of the Fagnano problem, which means $\triangle PQC$ has the minimal perimeter among all inscribed triangles in $\triangle AMN.$ Take a point $L$ on $MN$ and inside the $\triangle XYZ$ (see Figure \ref{case2}). Since $\triangle LXY$  inscribed in $\triangle AMN$, we have $\text{per} (\triangle LXY) > \text{per} (\triangle PQC)=p_F.$ Since $\triangle LXY$ is contained in $\triangle XYZ$, therefore, using Fact 2, we have
			\begin{equation}
				\text{per} (\triangle XYZ) > \text{per} (\triangle LXY) > p_F
			\end{equation}
		\end{description}

		\begin{description}
			\item[\textbf{Case 2.2}] Suppose $Z$ lies between $CD$ \\
			The proof is similar to the Case 2.1, and we have 
			\begin{equation}
				\text{per} (\triangle XYZ) > \text{per} (\triangle OXY) > p_F
			\end{equation}
			\begin{figure}[H]
				\centering
				\includegraphics[height=5cm]{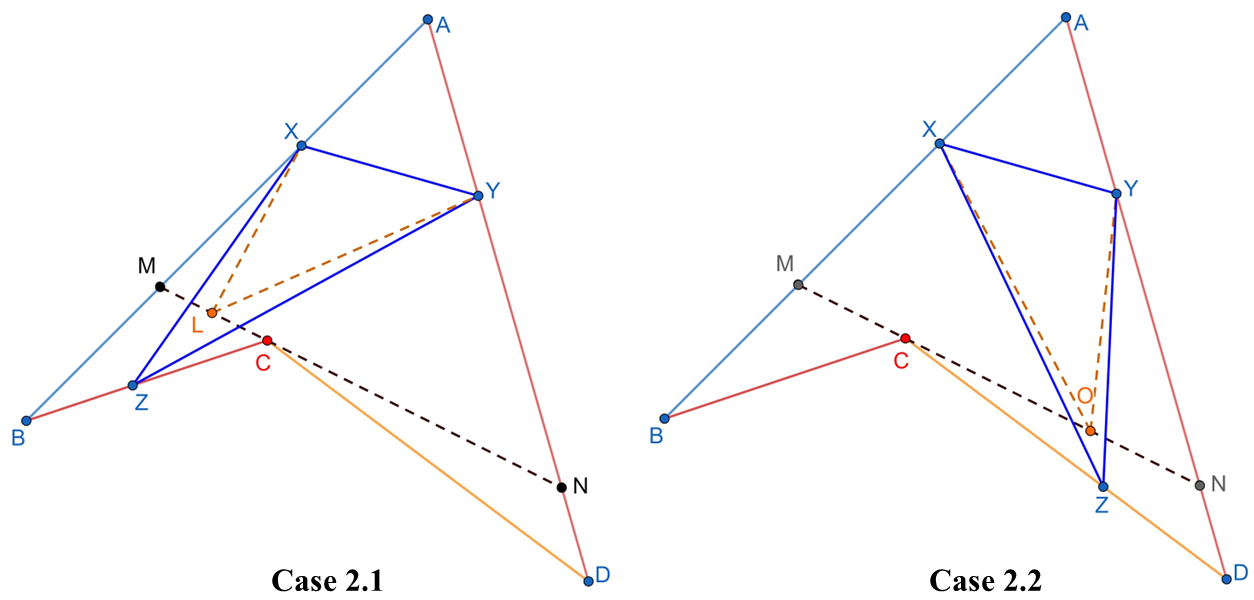}
				\caption{Illustration of Case 2}
				\label{case2}
			\end{figure}
		\end{description}
	\end{description}

	\begin{description}
		\item[\textbf{Case 3}] Suppose exactly one of the following holds: either $X$ lies between $AM$ or $Y$ lies between $AN$, and $Z$ lies on either $BC$ or $CD$. \\
		Let the segments $XY$ and $MN$ intersect at $T$, and given that $\triangle AMN$ is acute, it follows that any triangle  inscribed within it (different from $\triangle PQC$) will have a perimeter greater than that of the orthic triangle. Therefore, we have:
		$$p_F < \text{per}(\triangle YMT)$$
		Furthermore, since $\angle TMX (=\angle NMX) > 90^\circ$, we have $|TX| > |MT|$ and $|YX| > |YM|$  (see Figure \ref{case3}). Consequently, we have:
		\begin{align*}
			p_F & < \text{per}(\triangle YMT) \\
			&= |YM| +| MT| + |TY| \\
			& < |YX |+ |TX| + |TY| \\
			&= |YX| + |YX| \\
			& < |YX| +| XZ| + |YZ| \\
			&= \text{per}(\triangle XYZ) \\
			\text{i.e., } p_F & < \text{per}(\triangle XYZ)
		\end{align*}
		\begin{figure}[H]
			\centering
			\includegraphics[height=6cm]{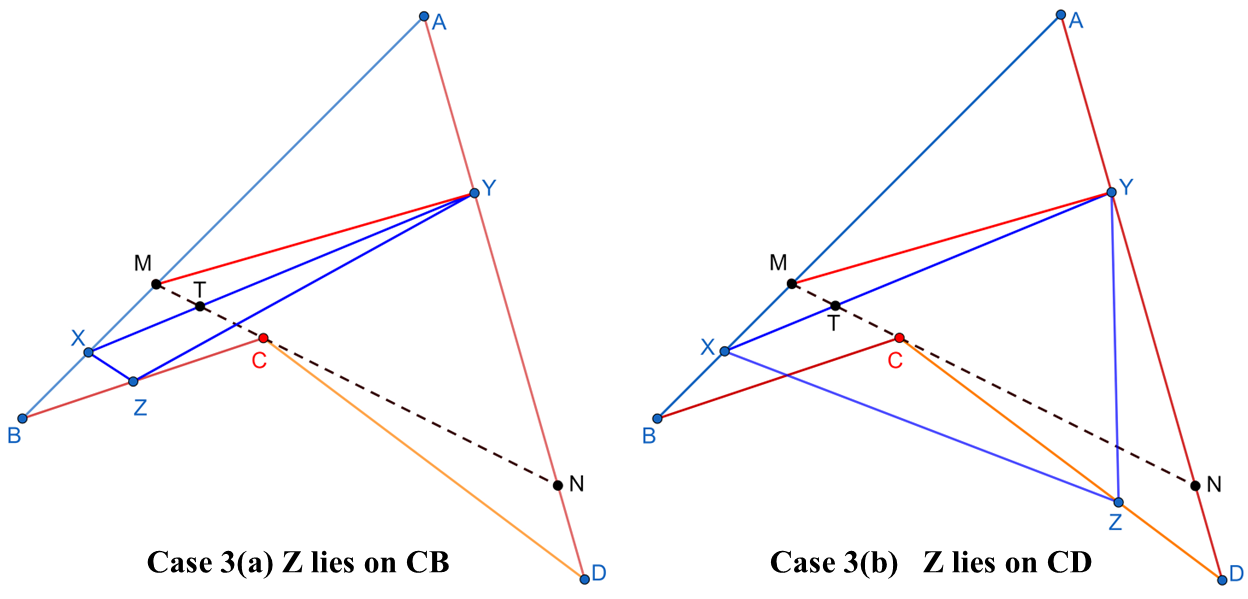}
			\caption{Illustration of Case 3}
			\label{case3}
		\end{figure}
	\end{description}
	
	\begin{description}
		\item[\textbf{Case 4}] Suppose both of the point  $X$ and $Y$ lie between $MB$ and $ND$ respectively, and $Z$ lies on either $BC$ or $CD$. \\
		Since the sum of two sides is always greater than the third side in $\triangle XYZ$, therefore $|XZ| + |ZY| > |XY|$  (see Figure \ref{case4}). Thus,
\begin{align*}
|XZ| + |ZY| + |XY| 
&> 2|XY| \\
&> 2|MN| \\
&> p_F 
\quad \text{[from (\ref{inequalities01}) and (\ref{perimeter<2MN})]} \\
\text{Hence,}\ \mathrm{per}(\triangle XYZ) &> p_F.
\end{align*}

		\begin{figure}[H]
			\centering
			\includegraphics[height=6cm]{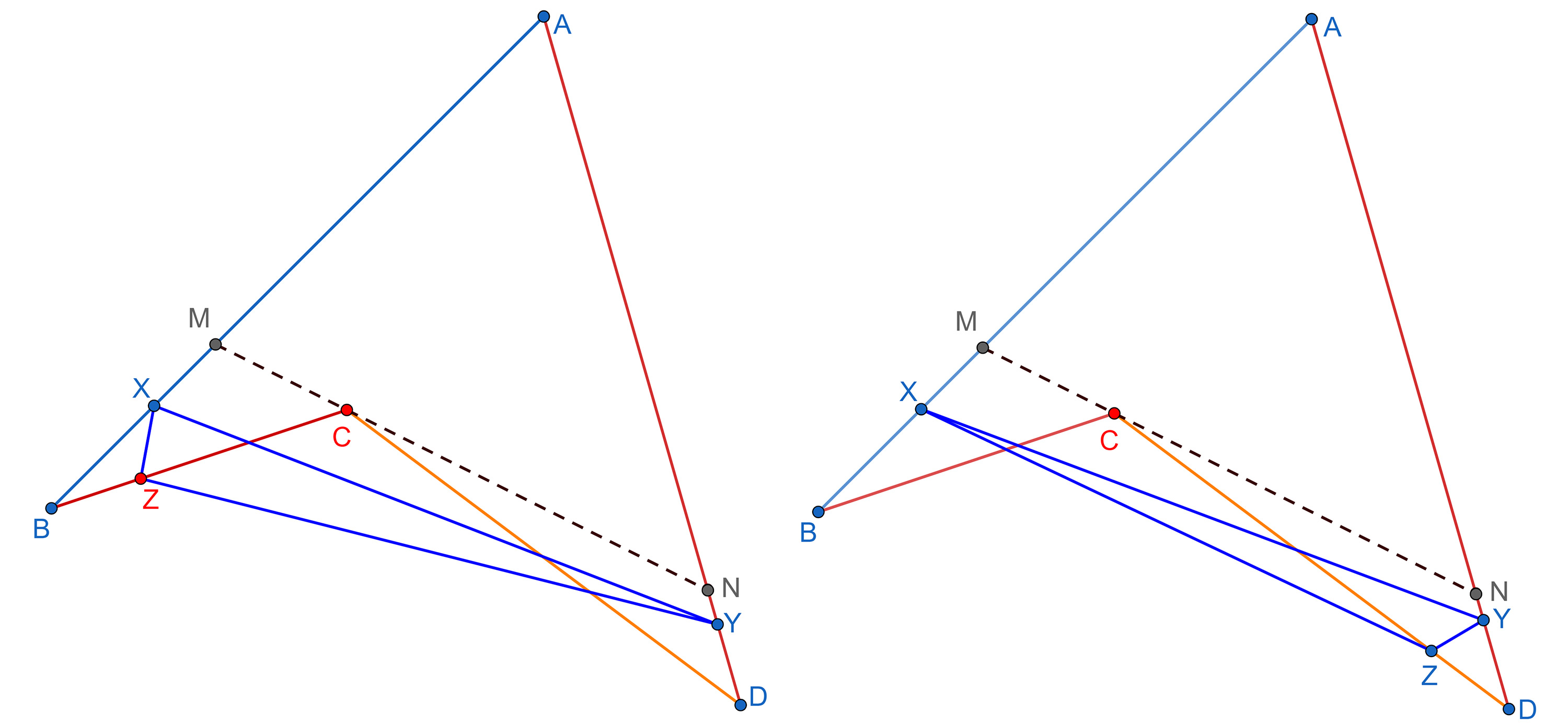}
			\caption{Illustration of Case 4}
			\label{case4}
		\end{figure}
	\end{description}

Hence, all possible configurations of the triangle $\triangle XYZ$ within the quadrangle $ABCD$ have been examined under the assumptions of Case~\textbf{(A)} in Theorem~\ref{maintheorem}. 
When the perpendicular to $AC$ through $C$ intersects both sides $AB$ and $AD$ at $M$ and $N$ respectively, then orthic triangle of the auxiliary triangle $\triangle AMN$ is the minimal perimeter triangle, to the generalized Fagnano problem.

\medskip

We now turn to Case~\textbf{(B)}, which arises when the angular conditions in equation~\eqref{existence} are not satisfied. 
Geometrically, this occurs when the perpendicular to $AC$ through $C$ fails to intersect exactly one of the sides $AB$ or $AD$( the nonconvexity ensures both the sides cannot be missed.)

\begin{description}
\item[\textbf{Case 5}] 
Without loss of generality, let us assume that the perpendicular 
 to $AC$ through $C$ fails to meet the side $AB$, but meets the side $AD$ at $N$. 
 Now extend the side $BC$ to meet $AD$ at $W$,  thereby forming the auxiliary triangle $\triangle ABW$. Note that the point $M$ is beyond $AB,$ therefore $N$ lies between $AW$. In $\triangle ACN,$  since  $\angle C = 90^{\circ}$,  $\angle  N < 90^{\circ}$. 
 Consequently, $\angle AWB < 90^\circ$. Therefore $\triangle ABW$ is acute-angled. We shall now demonstrate that the orthic triangle of $\triangle ABW$ yields the minimal-perimeter triangle for this case. We denote the perimeter by $p_W$. This situation gives rise to four subcases, discussed below.

\medskip
\noindent
\begin{description}
\item[\textbf{Case 5.1}] 
Suppose that $X$ lies on $AB$, $Y$ lies on $AW$, and $Z$ lies on $BC$. 
 
By the classical Fagnano problem, the orthic triangle of $\triangle ABW$ has the minimal perimeter$(p_W)$ among all triangles inscribed in $\triangle ABW$  (see Figure \ref{case5.1}). Hence, $p_W \leq \text{per} (\triangle XYZ)$ with equality if and only if $\triangle XYZ$ is the orthic triangle of $\triangle ABW$.
 
 \begin{figure}[H]
	\centering
	\includegraphics[height=6cm]{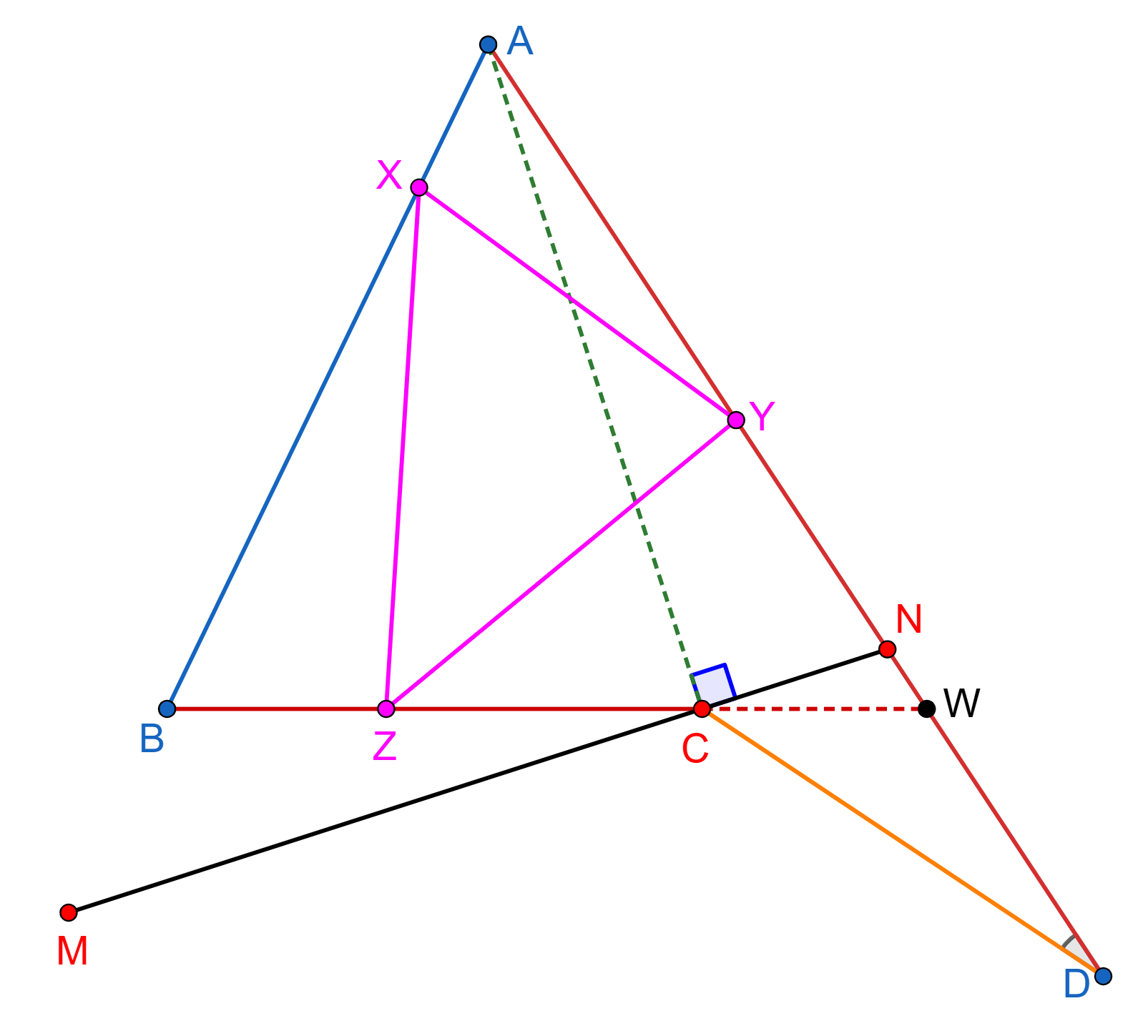}
	\caption{Illustration of Case 5.1}
	\label{case5.1}
\end{figure}

\item[\textbf{Case 5.2}] 

Suppose that $X$ lies on $AB$, $Y$ lies between $WD$, and $Z$ lies on $BC$.\\
Join $XW$. Since $\angle BWY > 90^\circ$, in $\triangle XWY$ we have  $|XY| > |XW|$, and in $\triangle ZWY$, we have  $|ZY| > |ZW|$  (see Figure \ref{case5.2}).
Therefore,
\begin{align}
	p_W< \text{per}(\triangle XZW) 
	&= |XZ| + |ZW| + |XW| \nonumber \\
	&< |XZ| + |ZY| + |XY| \nonumber \\
	&< \text{per}(\triangle XYZ). 
\end{align}

\begin{figure}[H]
	\centering
	\includegraphics[height=6cm]{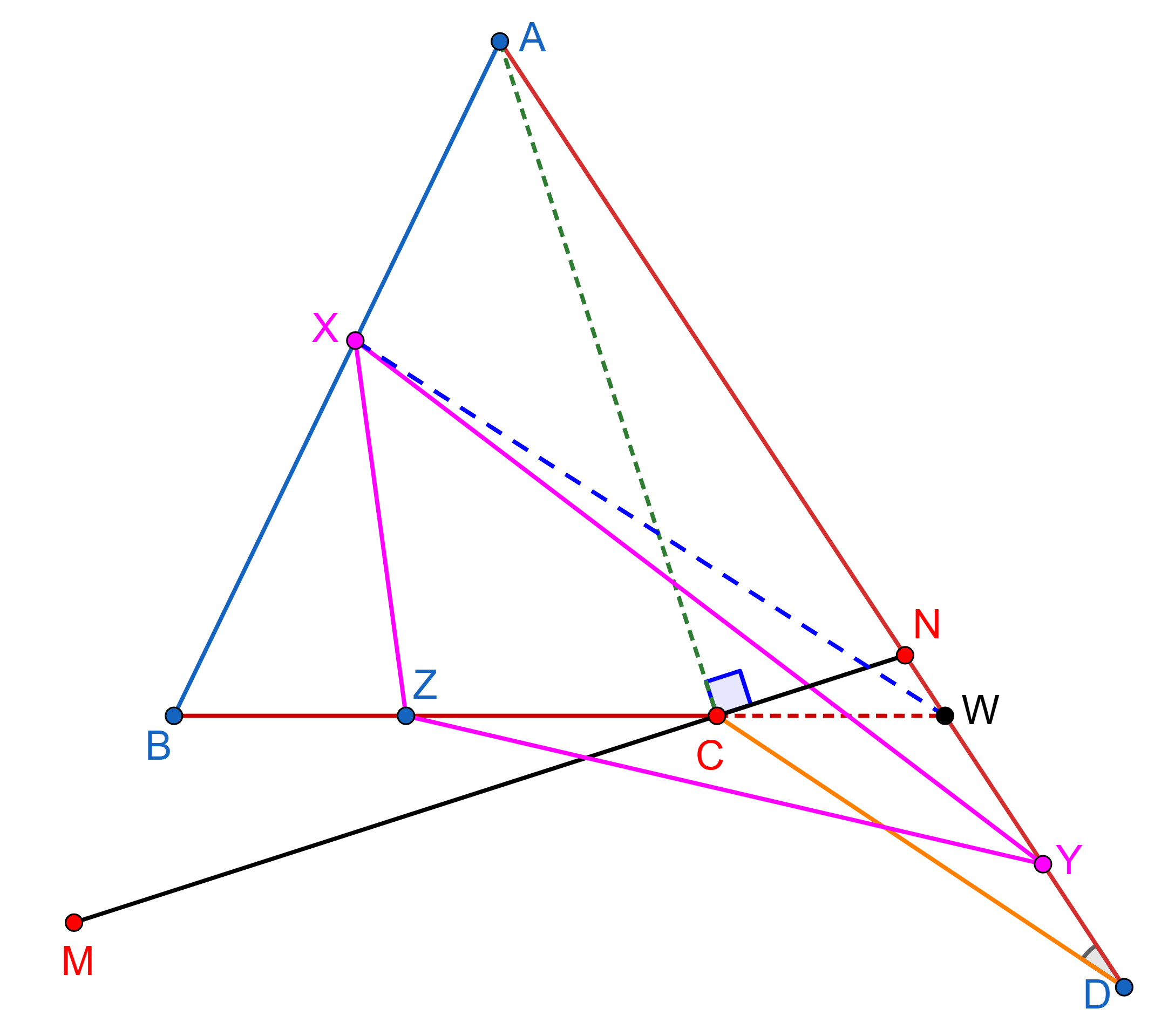}
	\caption{Illustration of Case 5.2}
	\label{case5.2}
\end{figure}

\item[\textbf{Case 5.3}] 
Suppose that $X$ lies on $AB$, $Y$ lies on $AW$, and $Z$ lies between $CD$. 
Let $XZ$ intersect $BW$ at $S$. Join $YS$. Since $\triangle XYS$ is inside the $\triangle XYZ$ with same base $XY$, using Fact 2, we have $p_{W} \leq \text{per}(\triangle XYS) <\text{per}( \triangle  XYZ)$  (see Figure \ref{case5.3}). Hence, result follows.

 \begin{figure}[H]
	\centering
	\includegraphics[height=6cm]{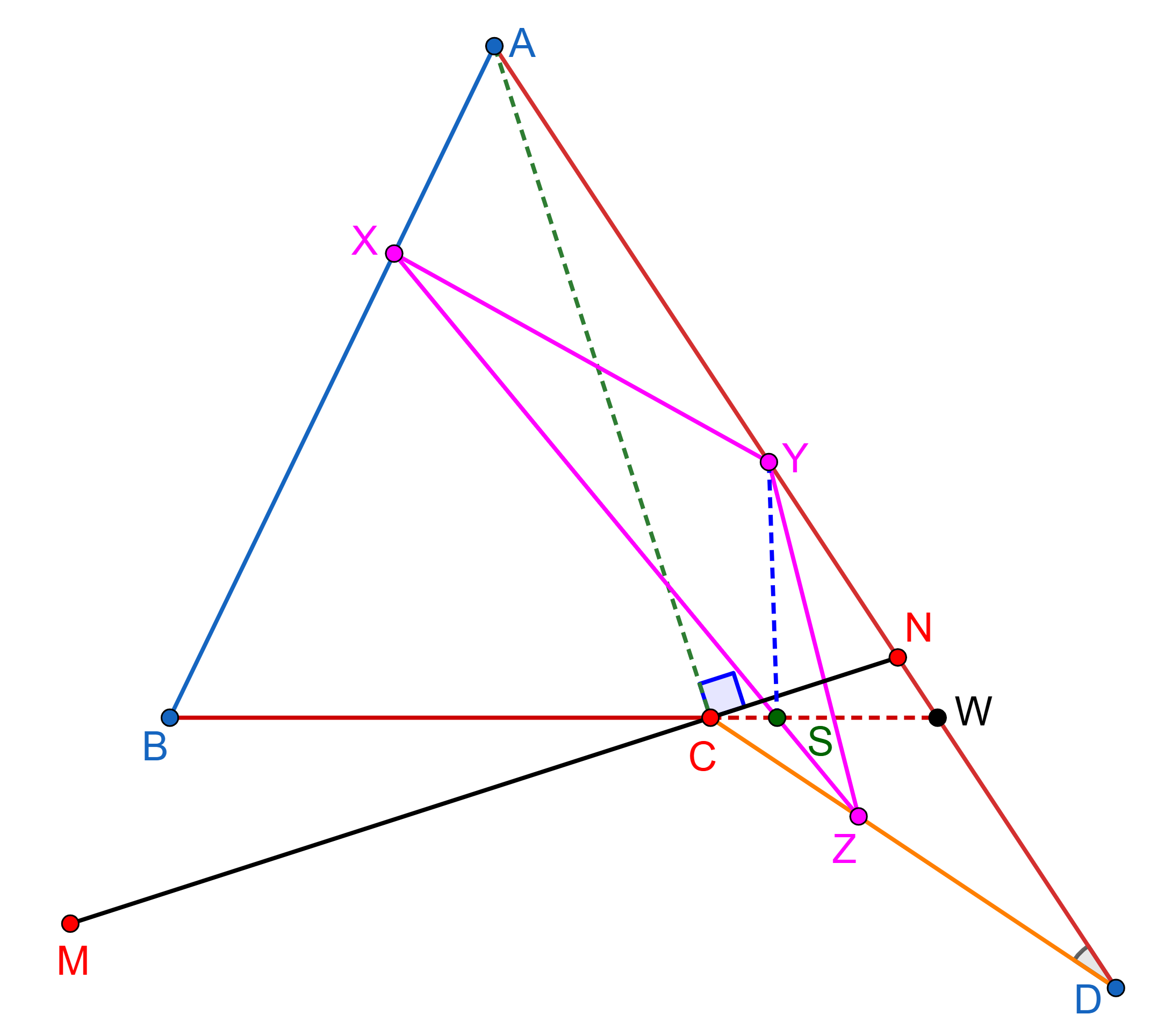}
	\caption{Illustration of Case 5.3}
	\label{case5.3}
\end{figure}

\item[\textbf{Case 5.4}]

Suppose that $X$ lies on $AB$, $Y$ lies between $WD$, and $Z$ lies on $CD$. \\

Let $XZ$ intersects $BW$ at $S$. Join $SY.$ In $\triangle SYW$  since $\angle SWY(=\angle BWY) > 90^\circ$, therefore $|SY| >|SW|$. 
In $ \triangle SZY, ~ |SZ| +|ZY| > |SY|$. Join $XW$. In $\triangle XYW$, $\angle XWY > 90^\circ$ whereby $|XY| > |WX|$  (see Figure \ref{case5.4}). It follows that

\begin{align}
	p_W < \text{per}(\triangle XSW)
	&< |XS| + |SW| + |WX| \notag\\
	& < |XS| + |SY|+ |WX| \notag\\
	&< |XS| + |SZ| + |ZY| + |WX| \notag\\
	&< |XZ| + |ZY| + |XY| = \text{per}(\triangle XYZ)\notag
\end{align}

\begin{figure}[H]
	\centering
	\includegraphics[height=6cm]{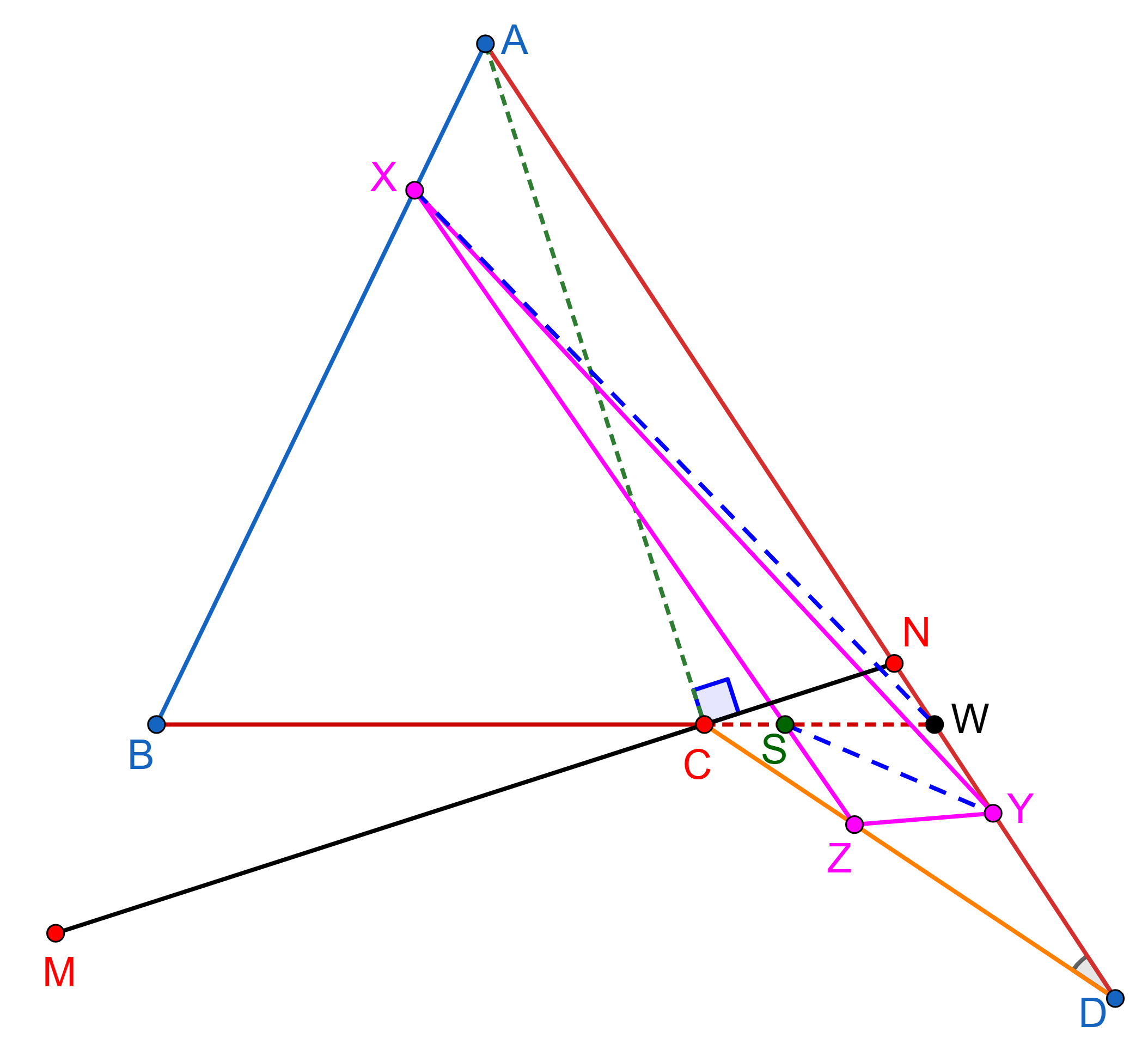}
	\caption{Illustration of Case 5.4}
	\label{case5.4}
\end{figure}

\medskip
Thus, in all subcases~5.1, 5.2, 5.3 and~5.4, the orthic triangle of the auxiliary triangle $\triangle ABW$ is the unique minimal-perimeter triangle for Case~\textbf{(B)} of Theorem~\ref{maintheorem}. 
\end{description}
\end{description}	
\end{proof}

The analyses of Cases~\textbf{(A)} and~\textbf{(B)} together establish that in every admissible configuration of an acute-angled nonconvex quadrangle, the triangle of minimal perimeter is realized as the orthic triangle of an appropriately constructed auxiliary triangle—either $\triangle AMN$ or $\triangle ABW$. 
This unified result extends the classical Fagnano problem to nonconvex quadrangle and reveals a consistent geometric principle underlying both cases.

\begin{remark}
	It follows from Lemma~\ref{lemma01} that the perimeter of the minimal triangle in the generalized Fagnano problem is bounded above by \(2|MN|\) (or $2|BW|$ in Case-B), where \(MN \perp AC\) in the nonconvex quadrangle \(ABCD\). 
	This result establishes a simple and practical upper bound: the minimal perimeter cannot exceed twice the length of \(MN\)  (or $2 |BW|$).
\end{remark}

\begin{acknowledgment}
	We sincerely thank the anonymous reviewer for the constructive comments 
	and valuable suggestions, which greatly improved the quality and clarity of 
	this manuscript.
	The second author gratefully acknowledges his supervisor, Dr.~R.~K.~Pandey, for providing the opportunity to pursue this research and for his continuous mentorship and support throughout.

\end{acknowledgment}

\section*{Disclosure Statement}
The author declares that there is no conflict of interest regarding the publication of this paper.

\bibliographystyle{plain}
\bibliography{References.bib}	

\end{document}